\documentclass[a4, 11pt]{amsart}
\usepackage{amsmath,amsthm,amsfonts,amssymb,amscd}
\newtheorem{theorem}{Theorem}[section]
\newtheorem{lemma}[theorem]{Lemma}
\newtheorem{question}[theorem]{Question}
\usepackage{xcolor}
\usepackage{centernot,tikz}

\newtheorem{definition}[theorem]{Definition}
\newtheorem{example}[theorem]{Example}

\newtheorem{prop}[theorem]{Proposition}
\newtheorem{corollary}[theorem]{Corollary}

\newtheorem{remark}[theorem]{Remark}

\numberwithin{equation}{section}
\begin{document}
\title[On the Closure of Absolutely Norm Attaining Operators]{On the Closure of Absolutely Norm Attaining Operators}
\author{G. Ramesh }
 \address{G. Ramesh, Department of Mathematics, IIT Hyderabad, Kandi, Sangareddy, Telangana- 502284, India.}
 \email{rameshg@math.iith.ac.in}
\author{Shanola S. Sequeira}
 \address{Shanola S. Sequeira, Department of Mathematics, IIT Hyderabad, Kandi, Sangareddy, Telangana- 502284, India.}
 \email{ma18resch11001@iith.ac.in}
\maketitle
\begin{abstract}
Let $H_1$ and $H_2$ be complex Hilbert spaces and $T:H_1\rightarrow H_2$ be a bounded linear operator. We say $T$ to be norm attaining, if there exists $x\in H_1$ with $\|x\|=1$ such that $\|Tx\|=\|T\|$. If for every closed subspace $M$ of $H_1$, the restriction $T|_{M}:M\rightarrow H_2$ is norm attaining then, $T$ is called absolutely norm attaining operator or $\mathcal{AN}$-operator. If we replace the norm of the operator by the minimum modulus $m(T)=\inf{\{\|Tx\|:x\in H_1,\; \|x\|=1}\}$, then $T$ is called the minimum attaining and the absolutely minimum attaining operator (or $\mathcal{AM}$-operator) respectively.

 In this article, we discuss about the operator norm closure of the $\mathcal{AN}$-operators.  We completely characterize operators in this closure and study several important properties. We mainly give the spectral characterization of the positive operators in this class and give the representation when the operator is normal. Later we also study the analogous properties for $\mathcal{AM}$-operators and prove that the closure of $\mathcal{AM}$-operators is same as that of the closure of $\mathcal{AN}$-operators. As a consequence, we prove similar results for operators in the norm closure of $\mathcal{AM}$-operators.
\end{abstract}


\section{Introduction}
It is always been a question of interest: for which pair of Banach spaces $X,Y$, the set of all norm attaining operators is dense in $\mathcal B(X,Y)$, the Banach space of bounded linear operators from $X$ into $Y$, with respect to the operator norm. This question created a lot of impact on the research in Functional Analysis and it was answered affirmatively for the Banach spaces when the co-domain is one dimensional by Bishop and Phelps \cite{BishopPhelps}. Lindenstrauss \cite{Lindenstrauss} proved that this result is also true for the Banach spaces when the domain is reflexive but not in general. We refer the expository paper by M. ACosta \cite{Acos denseness} for more details. Hence the set of all norm attaining operators is norm dense in the space of all bounded linear operators between the complex Hilbert spaces. Enflo et.al.\cite{EnfloKoverSmithies} gave a simple proof of this result. This motivates us to study the operator norm closure of a subclass of the set of all norm attaining operators, namely,  absolutely norm attaining operators which has been studied extensively in the recent times \cite{CarvjalNeves,PandeyPaulsen,RameshAN1,Rameshpara,venkuramesh}.

Since the class of $\mathcal{AN}$-operators contains compact operators, the closure of $\mathcal{AN}$-operators with respect to the strong operator topology is $\mathcal{B}(H)$, the space of all bounded linear operators defined on a Hilbert space $H$. In the present article we determine the closure of $\mathcal{AN}$-operators with respect to the operator norm of $\mathcal B(H)$ and we show that an operator is in the $\mathcal{AN}$-closure if and only if it is a compact perturbation of an $\mathcal{AN}$-partial isometry. This result helps us to look at a few important properties of elements in this class. We mainly characterize positive operators, self-adjoint operators and normal operators in this class and study their structures. We give a characterization of positive operators in the operator norm closure of $\mathcal{AN}$-operators in terms of the essential spectrum which is a generalization of a result in \cite{Rameshpara}. We also look at the characterization of positive $\mathcal{AN}$-operators found in \cite{venkuramesh} and that of positive $\mathcal{AM}$-operators found in \cite{GRS}, in a different way. Further, we study about a sufficient condition under which the adjoint of an operator in the closure of $\mathcal{AN}$-operators also belongs to the same class.

A  counterpart of  norm attaining operators is the class of minimum attaining operators. In this case also a Lindenstrauss type theorem is proved in \cite{SHKGRMIN}. Analogous to $\mathcal{AN}$-operators, we can define $\mathcal{AM}$-operators or the class of absolutely minimum attaining operators. A more detailed study of these operators can be found in \cite{NeeruRamesh,NBGR,GRS}. This class is a subset of the class of bounded operators with closed range, which distinguishes it  from that of $\mathcal{AN}$-operators. In this article we discuss the operator norm closure of this class of operators also. In fact, we conclude that the operator norm closure of $\mathcal{AN}$-operators and $\mathcal{AM}$-operators is the same. This fact allows  us to prove many results in the latter case with the help of the former one.  Results of this paper generalizes the corresponding results for $\mathcal{AN}$-operators as well as the $\mathcal{AM}$-operators, for example \cite{NeeruRamesh,NBGR,GRS,Rameshpara,venkuramesh}.

We organize the article as follows: In the second section we give the basic definitions and notations which are used throughout the paper. In the third section, we give the structure of general operators in the $\mathcal{AN}$-closure and some of their properties. In the fourth section we give a characterization of positive operators in closure of $\mathcal{AN}$-operators in terms of its essential spectrum and discuss the structure of positive operators. In the fifth section we discuss the structure theorem for self-adjoint and normal operators and in the sixth section we study the structure of operators in the $\mathcal{AM}$-closure and compare its properties with that of operators in $\mathcal{AN}$-closure.

\section{Definitions and Notations}
Throughout the paper, we let $H, H_1, H_2$ to denote  complex Hilbert spaces of arbitrary dimensions. A linear operator $T : H_1 \to H_2$ is said to be bounded if there exists a constant $k > 0$, such that $\|Tx\| \leq k \|x\|$, for all $x \in H_1$. The space of all bounded linear operators from $H_1$ to $H_2$ is denoted by $\mathcal{B}(H_1,H_2)$ and if $H_1 = H_2= H$, then we write this simply by $\mathcal{B}(H)$. The operator norm on $T \in \mathcal{B}(H_1, H_2)$ is given by,
\begin{equation*}
\|T\| = \sup\{\|Tx\| : x \in H_1, \|x\| = 1\}.
\end{equation*}
If $\mathcal{A} \subseteq \mathcal{B}(H_1, H_2)$, then we denote the closure of $\mathcal{A}$ with respect to the operator norm by $\overline{\mathcal{A}}$. The range and nullspaces of $T \in \mathcal{B}(H_1, H_2)$ are denoted by $R(T)$ and $N(T)$ respectively. Given $T \in \mathcal{B}(H_1, H_2)$, there exists a unique operator $T^* \in \mathcal{B}(H_2, H_1)$ such that
\begin{equation}
\langle Tx, y \rangle = \langle x, T^{*}y \rangle \quad \forall x \in H_1, y \in H_2. \nonumber
\end{equation}
The operator $T^*$ is called adjoint of $T$. If $T \in \mathcal{B}(H)$, then $T$ is called self-adjoint if $T = T^*$ and positive, if it is self-adjoint and $\langle Tx,x \rangle \geq 0$ for all $x \in H$. If $\mathcal{A} \subseteq \mathcal{B}(H)$, then the set of all positive operators in $\mathcal{A}$ is denoted by $\mathcal{A}_+$. An operator $T \in \mathcal{B}(H)$ is said to be normal if $T^*T = TT^*$. Let $T_1, T_2 \in \mathcal{B}(H)$ be self-adjoint. Then we say $T_1 \leq T_2$, if $T_2 - T_1 \in \mathcal{B}(H)_+$.

  If $T \in \mathcal{B}(H_1, H_2)$, then $T$ is said to be finite rank if $R(T)$ is finite dimensional and  compact if for every bounded set $B$ of $H_1$, $T(B)$ is a pre-compact set in $H_2$. The set of all finite rank and compact operators is denoted by $\mathcal{F}(H_1, H_2)$ and $\mathcal{K}(H_1, H_2)$ respectively. If $H_1 =H_2 = H$, then we denote them by  $\mathcal{F}(H)$ and $\mathcal{K}(H)$, respectively. If $M$ is a closed subspace of $H$, then we denote the orthogonal complement of $M$ by $M^\perp$ and the orthogonal projection onto $M$ by $P_{M}$.

   Given any positive operator $T \in \mathcal{B}(H)$, there exists a unique positive operator $S \in \mathcal{B}(H)$ such that $S^2 = T$. Then this operator $S$ is called the positive square root of $T$ and it is denoted by $T^{1/2}$.

   An operator $W \in \mathcal {B}(H_1, H_2)$ is called partial isometry, if
 \begin{equation}
 \|Wx\| = \|x\| ,\quad \forall x \in {N(W)}^{\perp}. \nonumber
 \end{equation}
  If $T \in \mathcal{B}(H_1, H_2)$, then $T^*T \in \mathcal{B}(H_1)_+$ and $|T| = (T^*T)^{1/2}$ is called the modulus of $T$. There exists a unique partial isometry $W \in \mathcal{B}(H_1, H_2)$ such that $T = W |T|$ and $N(W) = N(T)$. This is called the polar decomposition of $T$.

  	An operator $T \in \mathcal{B}(H)$ is said to be invertible if there exists unique $T^{-1} \in \mathcal{B}(H)$ such that $T T^{-1} = T^{-1}T = I$. We define the resolvent of $T$ denoted by $\rho(T)$ as
  \begin{equation*}
  \rho(T) = \{ \lambda \in \mathbb{C} : T - \lambda I \ \text{is}\; \text{invertible} \ \text{in} \ \mathcal{B}(H)\},
  \end{equation*}
  and $\sigma(T) := \mathbb{C} \setminus \rho(T)$ is defined as the spectrum of $T$.\\

  The spectrum of $T$ can be decomposed as the disjoint union of the point spectrum $\sigma_{p}(T)$, the residual spectrum $\sigma_{r}(T)$ and the continuous spectrum $\sigma_{c}(T)$ where,\\
  $\sigma_{p}(T):= \{\lambda \in \mathbb{C} : T-\lambda I \ \text{is}\; \text{not}\; \text{injective}\}$,\\
  $\sigma_{r}(T) := \{\lambda \in \mathbb{C} : T-\lambda I \ \text{is}\; \text{injective}\; \text{but}\; \overline{R(T-\lambda I)} \neq H\}$,\\
  $\sigma_{c}(T) := \sigma(T) \setminus (\sigma_{p}(T) \cup \sigma_{r}(T))$.

  The above basic details of opertor theory is found in \cite{conway,Gohberg,Halmosprobs,Kubrusly}.

\begin{definition}\cite[Definition 1.1, 1.2]{CarvjalNeves}
  An operator $T \in \mathcal{B}(H_1,H_2)$ is said to be norm attaining if there exists $x \in H_1$ with $\|x\| = 1$ such that $\|Tx\| = \|T\|$. If for every closed subspace $M$ of $H_1$, $T|_M : M \to H_2$ is norm attaining then $T$ is said to be absolutely norm attaining or $\mathcal{AN}$- operator.
\end{definition}
The set of all norm attaining and absolutely norm attaining operators from $H_1$ to $H_2$ is denoted by $\mathcal{N}(H_1, H_2)$ and $\mathcal{AN}(H_1,H_2)$ respectively. If $H_1= H_2 = H$, then we denote them by $\mathcal{N}(H)$ and $\mathcal{AN}(H)$ respectively.

More details and examples of $\mathcal{AN}$-operators can be found in \cite{CarvjalNeves, RameshOsaka1, PandeyPaulsen,RameshAN1,Rameshpara,venkuramesh}.

Analogues to $\mathcal{AN}$-operators, we have another class of operators called the Absolutely minimum attaining operators. In order to discuss this class of operators, we need the following definition of the minimum modulus of $T \in \mathcal{B}(H_1,H_2)$.
  \begin{equation*}
  m(T):= \inf\{\|Tx\| : x \in H_1, \|x\| = 1 \}.
  \end{equation*}
  \begin{definition}\cite[Definition 1.1, 1.4]{CarvajalNevesAM}
  	An operator $T \in \mathcal{B}(H_1, H_2)$ is said to be minimum attaining if there exists $x_0 \in H_1$ with $\|x_0\| = 1$ such that $m(T) = \|Tx_0\|$. If for every closed subspace $M$ of $H_1$, $T|_M : M \to H_2$ is minimum attaining then $T$ is said to be absolutely minimum attaining or $\mathcal{AM}$- operator.
  \end{definition}
The set of all $\mathcal{AM}$-operators from $H_1$ to $H_2$ is denoted by $\mathcal{AM}(H_1,H_2)$ and when $H_1=H_2=H$, we write $\mathcal{AM}(H,H) = \mathcal{AM}(H)$.

More details of $\mathcal{AM}$-operators can be found in \cite{ NeeruRamesh, CarvajalNevesAM,GRS}.
\begin{definition}\cite[Definition 2.1, Page 349]{conway}\label{defFredholm}
	If $T \in \mathcal{B}(H_1, H_2)$, then $T$ is said to be left semi- Fredholm if there exists $A \in \mathcal{B}(H_2, H_1)$ and $K_1 \in \mathcal{K}(H_1)$, such that $AT = K_1 + I$ and right semi-Fredholm if there exists $B \in \mathcal{B}(H_2,H_1)$ and $K_2 \in \mathcal{K}(H_2)$ such that $TB = K_2 + I$.
	
	If $T$ is both left semi-Fredholm and right semi-Fredholm then it is called Fredholm.
\end{definition}
\begin{definition}\label{ess_spectrum}\cite[Definiton 4.1, Page 358]{conway}
	Let $H$ be an infinite dimensional Hilbert space. For $T \in \mathcal{B}(H)$, the essential spectrum of $T$ defined as,
$\sigma_{ess}(T) = \sigma(\pi(T))$, where $\pi : \mathcal{B}(H) \to \mathcal{B}(H)/ \mathcal{K}(H)$ is the quotient map.
\end{definition}
For $ T = T^* \in \mathcal{B}(H)$, the spectrum of $T$ is the disjoint union of discrete and essential spectum of $T$, where the discrete spectrum $\sigma_{d}(T)$ is the set of all eigenvalues of $T$ of finite multiplicities which are isolated points of the spectrum $\sigma(T)$. The complement
  	set $ \sigma_{ess}(T) := \sigma(T) \setminus \sigma_{d}(T)$ is the essential spectrum of $T$ \cite[Definition 8.3, Page 178]{Konrad}.

  There is an another equivalent definition for the essential spectral point. This is explained in the following theorem.
  \begin{theorem}\label{ess spectrum of self-adjoint}\cite[Theorem VII.11, Page 236]{ReedSimon1}
	Let $H$ be an infinite dimensional Hilbert space and $T = T^* \in \mathcal{B}(H)$. Then $\lambda \in \sigma_{ess}(T)$ if and only one or more of the following conditions hold.
\begin{enumerate}
	\item $\lambda$ is an eigenvalue of infinite multiplicity.
	\item $\lambda$ is the limit point of $\sigma_{p}(T)$.
	\item $\lambda \in \sigma_{c}(T)$.
\end{enumerate}
\end{theorem}
More details on essential spectrum can be found in \cite{conway,Muller,ReedSimon1}.  There are various definitions for the essential spectrum of a bounded operator in the literature. But all those definitions coincide when the underlying operator is self-adjoint. We refer to \cite[Section 3.14]{BarrySimon4} for this discussion and a few more historical account.
\begin{definition}\cite{Bouldin}
For $T \in \mathcal{B}(H)$, the essential minimum modulus of $T$ is defined as
\begin{equation*}
m_e(T) = \inf \{\lambda : \lambda \in \sigma_{ess}(|T|)\}.
\end{equation*}
\end{definition}

\begin{definition}\label{self-adjointdecomposition}\cite[Proposition 4.4.9, Page 158]{Pedersen}
	Let $T = T^*$. We define $T^+:=\frac{|T|+T}{2}$ and $T^-:= \frac{|T|-T}{2}$. Then $T = T^+ - T^-$. Note that both $T^{+}$ and $T^{-}$ are positive and $T^{+}T^{-}=0$.
\end{definition}
\begin{definition}\cite[Page 30]{conway}
	Let $\{H_i\}_{i \in \mathbb{N}}$ be the family of Hilbert Spaces and  $T_i \in \mathcal{B}(H_i)$ for all $i \in \mathbb{N}$.  Let $H  = \displaystyle{\bigoplus_{n \in \mathbb{N}}} H_n$. Then the direct sum $T = \displaystyle{\bigoplus_{n \in \mathbb{N}}} T_n$ on $H$ is defined as,
	\begin{equation*}
	T(x_1, x_2,\dots) = (T_1x_1, T_2x_2,\dots) \ \forall \ x_i \in H_i, i \in \mathbb{N} .
	\end{equation*}
If $\displaystyle{\sup_{i \in \mathbb{N}}} \{\|T_i\|\} < \infty$ then, $T \in \mathcal{B}(H)$ and $\|T\| = \displaystyle{\sup_{i \in \mathbb{N}}}\{\|T_i\|\}$.
\end{definition}
 \section{Closure of $\mathcal{AN}$-operators}
 First we recall a characterisation of positive $\mathcal{AN}$-operators which we will be using in our results.
 	\begin{theorem} \cite[Theorem 5.1]{PandeyPaulsen} \label{PandeypositiveAN}
	Let $H$ be a complex Hilbert space of arbitrary dimension and let $T \in \mathcal{B}(H)_+$. Then $T \in \mathcal{AN}(H)$ iff  $T = \alpha I + K + F$, where $\alpha \geq 0$, $K \in \mathcal{K}(H)_+$ and $F \in \mathcal{F}(H)$ is self-adjoint.
\end{theorem}
\begin{theorem} \cite[Corollary 2.11]{venkuramesh} \label{T*T AN}
Let $T \in \mathcal{B}(H_1 , H_2 )$. Then $T \in \mathcal{AN}(H_1 , H_2)$ if and only if $T^*T \in  \mathcal{AN}(H_1)$.
\end{theorem}
	
Now we discuss some properties of $\mathcal{AN}$-operators which are useful in determining the closure of it.

\begin{prop}\label{AN+Finiterank}
	If $F \in \mathcal{F}(H_1,H_2)$ and $T \in \mathcal{AN}(H_1,H_2)$, then $T+F \in \mathcal{AN}(H_1,H_2)$.
\end{prop}
\begin{proof}
	Let $A = T+F$. Then $A^*A = (T^*+F^*)(T+F)= T^*T+F_1$, where $F_1 = T^*F+F^*T+F^*F \in \mathcal{F}(H_1)$ and $F^*_1=F_1$. By using Theorem \ref{T*T AN} and Theorem \ref{PandeypositiveAN}, we have $T^*T = \alpha I + K + F_2$, where $\alpha \geq 0, K \in \mathcal{K}(H_1)_+$ and $F_2 =F^*_2 \in \mathcal{F}(H_1)$. Then $A^*A = \alpha I + K + \tilde{F}$, where $\tilde{F} = F_2 + F_1$ and ${\tilde{F}}^* =\tilde{F}$. So by Theorem \ref{PandeypositiveAN}, we get $A^*A \in \mathcal{AN}(H_1)$. Hence by Theorem \ref{T*T AN},  $A \in \mathcal{AN}(H_1,H_2)$.
\end{proof}
\begin{corollary}\label{partialisometry+finiterank}
 If $\alpha \geq 0$, $F \in \mathcal{F}(H_1,H_2)$ and $ W \in \mathcal{AN}(H_1,H_2)$ is a partial isometry, then $\alpha W + F \in \mathcal{AN}(H_1,H_2)$.
 \end{corollary}
\begin{prop}\label{AN+compact}
	If $K \in \mathcal{K}(H_1,H_2)$ and $T\in \mathcal {AN}(H_1,H_2)$, then $T+K \in \overline{\mathcal{AN}(H_1,H_2)}$.
\end{prop}
\begin{proof}
 Since $K \in \mathcal{K}(H_1,H_2)$, there exists a sequence $\{F_n\} \subset \mathcal{F}(H_1,H_2)$ such that $F_n \to K$ as $n \to \infty$. Then,
 \begin{equation}
   T + K = \lim_{n \to \infty}(T + F_n). \nonumber
  \end{equation}
  From Proposition \ref{AN+Finiterank}, we have $T + F_n \in \mathcal{AN}(H_1,H_2)$ for all $ n \in \mathbb{N}$. Hence $T+K \in \overline{\mathcal{AN}(H_1,H_2)}.$
\end{proof}
\begin{corollary}\label{partialisometryAN+compact}
	If $A = \alpha W + K$, where $\alpha \geq 0$, $W \in \mathcal{AN}(H_1,H_2)$ is a partial isometry and $K \in \mathcal{K} (H_1,H_2)$, then $A \in \overline{\mathcal{AN}(H_1,H_2)}$.
\end{corollary}
It is observed that the converse of Corollary \ref{partialisometryAN+compact} is also true. Hence we give the characterization for the operators in $\overline{\mathcal{AN}(H_1, H_2)}$.
\begin{theorem}\label{general AN-closure structure}
Let $T \in \overline{\mathcal{AN}(H_1,H_2)}$ and $T = W|T|$ be the polar decomposition of $T$. Then $T = \alpha W + K$, for some $K \in \mathcal{K}(H_1,H_2)$ and $\alpha \geq 0$.
\end{theorem}
\begin{proof}
Since $T \in \overline{\mathcal{AN}(H_1,H_2)}$, there exists a sequence $\{T_n\}$ of $\mathcal{AN}$-operators, such that $T_n \to T$ as $n \to \infty$ in the operator norm. Then $|T_n| \to |T|$ as $n \to \infty$ by \cite[Problem 15, Page 217]{ReedSimon1}. As $T_n \in \mathcal{AN}(H_1,H_2)$, we have $|T_n|\in \mathcal{AN}(H_1)$. Hence by Theorem \ref{PandeypositiveAN}, we get $|T_n|= \alpha_n I + K_n + F_n$ where $\alpha_n \geq 0, \ K_n \in \mathcal{K}(H_1)_+$ and $F^*_n= F_n \in \mathcal{F}(H_1)$ for all $n \in \mathbb{N}$. As $\alpha_n \in \sigma(|T_n|)$, we have $\alpha_n\leq \||T_n|\|=\|T_n\|<\displaystyle{\sup_{n}}\|T_n\|<\infty$, so $\{\alpha_n\}$ is bounded. Hence there exists a subsequence say, $\{\alpha_{n_k}\}$ such that $\alpha_{n_k} \to \alpha$ as $k \to \infty$ for some $\alpha \geq 0.$ Then $K_{n_k} + F_{n_k} \to K_1$, for some $K^*_1 = K_1 \in \mathcal{K}(H_1)$. Hence $|T_{n_k}| \to \alpha I + K_1$ as $k \to \infty$. This implies $|T| = \alpha I + K_1$. Thus $T = \alpha W + K$, where $K = W K_1 \in \mathcal{K}(H_1,H_2)$.
\end{proof}

\begin{corollary}\label{properties2}
Let $T\in \overline{\mathcal{AN}(H_1,H_2)}$ and $T\notin \mathcal{K}(H_1,H_2)$ . Then
\begin{enumerate}
  \item \label{fdmlkernel} $N(T)$ is finite dimensional	and $W\in \mathcal{AN}(H_1,H_2)$
\item \label{closedrange}  $R(T)$ is closed
\item \label{leftsemifredholm} $T$  is left semi-Fredholm operator.
\end{enumerate}
\end{corollary}

\begin{proof}
From Theorem \ref{general AN-closure structure}, we have
\begin{equation}\label{propeq1}
|T| = \alpha I + K_1,\, \text{ for some}\; \alpha \geq 0\; \text{and}\; K_1 \in \mathcal{K}(H).
\end{equation}
Proof of (\ref{fdmlkernel}):  If $x \in N(T)$, then $x \in N(|T|)$ and $\alpha x = -K_1x$. Hence $\alpha I|_{N(|T|)}$ is a compact operator. This implies $N(|T|)$ is finite dimensional. Hence $N(T)$ is finite dimensional. Also, since $N(|T|) = N(W)$ we get, $N(W)$ is finite dimensional and $W \in \mathcal{AN}(H_1,H_2)$ \cite[Proposition 3.14]{CarvjalNeves}.

Proof of (\ref{closedrange}): Since $T \notin \mathcal{K}(H_1, H_2)$ we have $\alpha>0$. Hence by Equation \eqref{propeq1},  $R(|T|)$ is closed, consequently,  $R(T)$ is closed.

Proof of (\ref{leftsemifredholm}) Let $T = W|T|$ be the polar decomposition of $T$. Then $|T| = W^*T= \alpha I + K_1$, by Equation \eqref{propeq1}. Hence by Definition \ref{defFredholm}, $T$ is left semi-Fredholm.
\end{proof}
\begin{corollary}\label{property3}
Let $V\in \mathcal B(H_1,H_2)$ be a partial isometry. Then $V\in \mathcal{AN}(H_1,H_2)$ if and only if $V\in \overline{\mathcal{AN}(H_1,H_2)}$.
\end{corollary}
\begin{proof}
	If $V \in \mathcal{AN}(H_1, H_2)$, then clearly $V \in \overline{\mathcal{AN}(H_1, H_2)}$.
	
	Conversely let $V \in \overline{\mathcal{AN}(H_1,H_2)}$. If $V \in \mathcal{K}(H_1, H_2)$, then $V \in \mathcal{AN}(H_1,H_2)$. If $V \notin \mathcal{K}(H_1, H_2)$, then by Corollary \ref{properties2}(\ref{fdmlkernel}) we have, $N(V)$ is finite dimensional. Hence $V \in \mathcal{AN}(H_1,H_2)$ by \cite[Proposition 3.14]{CarvjalNeves}.
\end{proof}

Next we discuss some of the properties of operators in the set  $\overline{\mathcal{AN}(H)}$.
\begin{lemma}\cite[Theorem 2.2]{RameshOsaka1} \label{productinAN}
	If $A, B \in \mathcal{AN}(H)$, then $AB \in \mathcal{AN}(H)$.
\end{lemma}
\begin{theorem}\label{productinclosure}
	If $A, B \in \overline{\mathcal{AN}(H)}$ then, $AB \in \overline{\mathcal{AN}(H)}$.
\end{theorem}
\begin{proof}
	Since $A, B \in \overline{\mathcal{AN}(H)}$, there exist sequences  $\{A_n\}$ and $\{B_n\}$ of $\mathcal{AN}$- operators such the $A_n \to A$ and $B_n \to B$ as $n \to \infty$. By Lemma \ref{productinAN}, we have $A_nB_n \in \mathcal{AN}(H)$. Also $A_nB_n \to AB$ as $n \to \infty$. Thus $AB \in \overline{\mathcal{AN}(H)}$.
\end{proof}
\begin{corollary}
  If $T\in \overline{\mathcal{AN}(H)}$, then $T^n\in \overline{\mathcal{AN}(H)}$ for each $n\in \mathbb N$. In particular, if $T$ is positive and $p$ is real polynomial with positive co-efficients, the $p(T)\in \overline{\mathcal{AN}(H)}_{+}$.
\end{corollary}
\begin{lemma}\label{|T|AN-closure}
	Let $T \in \mathcal{B}(H_1,H_2)$.  Then  $T \in \overline{\mathcal{AN}(H_1,H_2)}$ if and only if $|T| \in \overline{\mathcal{AN}(H_1)}$.
\end{lemma}
\begin{proof}
	Let $T = W|T|$ and $|T| = P_{N(|T|)^\perp}|T|$ be the polar decompositions of $T$ and $|T|$ respectively.  First assume that $|T|\in \overline{\mathcal{AN}(H_1)} $.

Then by Theorem \ref{general AN-closure structure}, we have $|T| = \alpha P_{N(|T|)^\perp} + K$ for some $K \in \mathcal{K}(H_1)$, $\alpha \geq 0$. If $\alpha=0$, then $|T|\in \mathcal{K}(H_1)$ and hence $T\in \mathcal{K}(H_1,H_2) \subset \overline{\mathcal{AN}(H_1, H_2)}$.

Next assume that $\alpha>0$. Then  $N(|T|)$ is finite dimensional. Since $N(|T|) = N(W)$, we get $W \in \mathcal{AN}(H_1,H_2)$ \cite[Proposition 3.14]{CarvjalNeves}. Hence by using Theorem \ref{productinclosure}, we have $T= W|T| \in \overline{\mathcal{AN}(H_1, H_2)}$.

	Conversely, if $T \in \overline{\mathcal{AN}(H_1,H_2)}$, then there exists a sequence ${\{T_n}\}$ of $\mathcal{AN}$-operators which converges to $T$ in the operator norm. Then $|T_n|$ converges to $|T|$. Since $|T_n| \in \mathcal{AN}(H_1)$ we get, $|T|\in \overline{\mathcal{AN}(H_1)}$.
\end{proof}

The following result generalizes that of \cite[Corollary 2.11]{venkuramesh}.
\begin{theorem} \label{T*T in AN-closure}
	Let $T \in \mathcal{B}(H_1,H_2)$. Then $T \in \overline{\mathcal{AN}(H_1,H_2)}$ if and only $T^*T \in \overline{\mathcal{AN}(H_1)}$.
\end{theorem}
\begin{proof}
		If $T = W|T|$ be the polar decomposition of $T$, then by Theorem \ref{general AN-closure structure}, $T = \alpha W +K$ for some $\alpha \geq 0$ and $K \in \mathcal{K}(H_1,H_2)$. Note that if $\alpha=0$, then $T$ is compact and hence $T^*T$ is compact. Hence $T^*T\in \overline{\mathcal{AN}(H_1)}$.

Next assume that $\alpha>0$. Then by Corollary \ref{properties2}, $N(T)$ is finite dimensional. We have $T^*T = \alpha^2 W^*W + K_1$, where $K_1 = \alpha W^*K + \alpha K^*W + K^*K \in \mathcal{K}(H_1)$. Hence $T^*T = \alpha ^2 P_{{N(W)}^ \perp}+ K_1$.  Since, $N(W)$ is finite dimensional, $P_{{N(W)}^ \perp} \in \mathcal{AN}(H_1)$ by using \cite[Theorem 3.9]{CarvjalNeves}. Hence $T^*T \in \overline{\mathcal{AN}(H_1)}$.

	Conversely, if $T^*T = |T|^2 \in \overline{\mathcal{AN}(H_1)}$, then there exists a sequence $\{S_n\}$  of positive $\mathcal{AN}$-operators  such that $S_n \to |T|^2$ as $n \to \infty$. By \cite[Problem 14.(a), Page 217]{ReedSimon1} or \cite[Lemma 1]{kaufman}, we have $S^{1/2}_n$ converges to $|T|$. Thus $|T| \in \overline{\mathcal{AN}(H_1)}$, by \cite[Corollary 2.10]{venkuramesh}. Hence by Lemma \ref{|T|AN-closure}, $T \in \overline{\mathcal{AN}(H_1,H_2)}$.
\end{proof}
\section{Positive Operators in $\overline{\mathcal{AN}(H)}$}
By \cite[Theorem 2.5]{venkuramesh}, it is known that if $H$ is an infinite dimensional Hilbert space and $T\in \mathcal B(H)$, then $T\in \mathcal{AN}(H)_{+}$  if and only if there exists a unique triple $(\alpha, K, F)$ where $K\in \mathcal{K}(H)_{+},F\in \mathcal{F}(H)_{+}$ and $\alpha \in [0,\infty)$ such that $T=\alpha I+K-F$ with $KF = 0$, $F \leq \alpha I$. Here we express the same result in a different way.
\begin{theorem}
	Let $H$ be an infinite dimensional Hilbert space and $T \in \mathcal{B}(H)$. Then the following are equivalent.
	\begin{enumerate}
		\item $ T \in \mathcal{AN}(H)_{+}$.
		\item There exists $\alpha \geq 0$, such that $(T-\alpha I)^+ \in \mathcal{K}(H)_+$ and $(T - \alpha I)^- \in \mathcal{F}(H)_+$ and $\|(T-\alpha I)^{-}\|\leq \alpha$.
	\end{enumerate}
\end{theorem}
\begin{proof}
	(1) $\implies$ (2). Let $T \in \mathcal{AN}(H)_+$. Then by \cite[Theorem 2.5]{venkuramesh} there exists a unique triple $(\alpha, K, F)$ such that $T = \alpha I + K- F$ \ where, $K \in \mathcal{K}(H)_+, F \in \mathcal{F}(H)_+, \alpha \geq 0, KF = 0$ and $F \leq \alpha I$. Clearly $T- \alpha I$ is self-adjoint. Define $T_\alpha := T - \alpha I$. Then we have,
	\begin{equation*}
	|T_\alpha|^2 = T^*_\alpha T_\alpha = T^2_\alpha = (K-F)^2 = K^2 + F^2=(K+F)^2,
	\end{equation*}
	as $KF = 0$.  Hence $|T_\alpha|^2 = (K+F)^2$. Since $K + F$ is positive, we have $|K+F| = K+F$. Hence by the uniqueness of the square root, we get $ |T_\alpha| = K + F$. Now it is easy to see that $ T^+_\alpha \in \mathcal{K}(H)_+$ and $T^-_\alpha \in \mathcal{F}(H)_+$ by using Definition \ref{self-adjointdecomposition}. Also, $\|(T-\alpha I)^{-}\|\leq \alpha$ by the representation of $T$.
	
	(2) $\implies$ (1): Assume that there exists $\alpha \geq 0$ such that $(T-\alpha I)^+ =: K \in \mathcal{K}(H)_+$, $(T-\alpha I)^{-} =: F \in \mathcal{F}(H)_{+}$ and $\|(T-\alpha I)^{-}\|\leq \alpha$. Then by using definition \ref{self-adjointdecomposition}, $T - \alpha I = K - F $. Thus $T = \alpha I - F + K \in \mathcal{AN}(H)$ by using Theorem \ref{PandeypositiveAN}. It is clear by definitions that $KF=0$ and $\|F\|\leq \alpha$.
\end{proof}

Now we give a new structure theorem for positive operators in $\overline{\mathcal{AN}(H)}$.

\begin{theorem}\label{T positive}
	Let H be an infinite dimensional Hilbert space. Then the following are equivalent.\begin{enumerate}
		\item $T \in \overline{\mathcal{AN}(H)}_{+}$.
		\item There exists $\alpha \geq 0, K_1,K_2 \in \mathcal{K}(H)_+$ such that
		\begin{equation}\label{repneq}
		T = \alpha I - K_1 + K_2\, \text{with} \, K_1 \leq \alpha I\; \text{ and}\; K_1 K_2 = 0.
		\end{equation}
		Moreover, $(T-\alpha I)^{+}=K_2$ and $(T-\alpha I)^{-}=K_1$, hence the representation of $T$ is unique.
	\end{enumerate}
\end{theorem}
\begin{proof}
	(1) $\implies$ (2)
	Since $T \in \overline{\mathcal{AN}(H)}_+$, there exists $\{T_n\} \subset \mathcal{AN}(H)_{+}$ such that $T_n \to T$ in the operator norm. By using \cite[Theorem 2.5]{venkuramesh},  we have, $T_n = \alpha_n I - F_n + K_n$ where, $K_n \in \mathcal{K}(H)_+, F_n \in \mathcal{F}(H)_+ , \alpha_n \geq 0, K_n F_n = 0$ and $F_n \leq \alpha_n I$ for all $n \in \mathbb{N}$. Since $\{\alpha_n\}$ is bounded, it has a convergent subsequence say $\{\alpha_{n_k}\}$ which converges to $\alpha$. Now consider, $T_{n_k} = \alpha_{n_k}I - F_{n_k} + K_{n_k}$. Clearly, $T_{n_k} - \alpha_{n_k}I = K_{n_k} - F_{n_k}$ is convergent. As $K_{n_k} F_{n_k} = 0$ we have,
	\begin{equation*}
	(K_{n_k} - F_{n_k})^2 = K^2_{n_k} + F^2_{n_k} = (K_{n_k} + F_{n_k})^2,\ \forall \ k \in \mathbb{N}.
	\end{equation*}
	Hence $K_{n_k} + F_{n_k}$ is convergent. Therefore we can conclude that both $\{K_{n_k}\}$ and $\{F_{n_k}\}$ are convergent. Let $F_{n_k} \to K_1$ and $K_{n_k} \to K_2$ as $ k \to \infty$. Since $F_{n_k}K_{n_k} = 0$ for all $k \in \mathbb{N}$, we get $K_1 K_2 = 0$. Also note that as $F_{n_k} \leq \alpha_{n_k} I$, we have $\|F_{n_k}\| \leq \alpha_{n_k}$ for all $k \in \mathbb{N}$. Hence $\|K_1\| \leq \alpha$ which implies $K_1 \leq \alpha I$. Therefore $T = \alpha I -K_1 + K_2$. It is easy to show that  $(T-\alpha I)^{+}=K_2$ and $(T-\alpha I)^{-}=K_1$ by the corresponding definitions.

Next, we show that the representation of $T$ is unique. Let $T=\beta I-\tilde{K_1}+\tilde{K_2}$, where $\tilde{K_1},\tilde{K_2}\in \mathcal{K}(H)_{+}$ with $\tilde{K_1}\tilde{K_2}=0$ and $\tilde{K_1}\leq \beta I$. Then by the Weyl's theorem for essential spectrum, we can conclude $\sigma_{ess}(T)={\{\alpha}\}={\{\beta}\}$. By the uniqueness of $(T-\alpha I)^{+}$ and $(T-\alpha I)^{-}$, we can get the required conclusions.

(2) $\implies$ (1)	Let $\{F_n\} \subset \mathcal{F}(H)_+$ be such that $F_n \to K_1$ in the operator norm and $F_n \leq \alpha I$ for all $n \in \mathbb{N}$. Then $T_n := \alpha I - F_n + K_2 \in \mathcal{AN}(H)_+$. Clearly $T_n$  converges to $T$ in the operator norm. Hence $T \in \overline{\mathcal{AN}(H)}$. It is clear that since $K_1 \leq \alpha$, we get $T\geq 0$.
\end{proof}
\begin{remark}
 For $T \in \overline{\mathcal{AN}(H)}_+, T = \alpha I - K_1 + K_2$ where, $\alpha \geq 0, K_1, K_2 \in \mathcal{K}(H)_+$ and $\|K_1\| \leq \alpha$. Then $\alpha I - K_1 \in \mathcal{AM}(H)_+$ by \cite[Theorem 5.8]{GRS}. Hence a positive operator in $\overline{\mathcal{AN}(H)}$ is a compact perturbation of positive $\mathcal{AM}$-operator.
\end{remark}

We have that the positive $\mathcal{AN}$-operators are diagonalizable by using \cite[Corollary 3.2]{PandeyPaulsen}. Next we prove that this result is also true for positive operators in $\overline{\mathcal{AN}(H)}$.

\begin{corollary}\label{diagonalization}
Let $T\in \overline{\mathcal{AN}(H)}_{+}$. Then $T$ is diagonalizable.
\end{corollary}
\begin{proof}
	By  Theorem \ref{T positive}, we get $T = \alpha I - K_1 + K_2$ where, $K_1, K_2 \in \mathcal{K}(H)_+$. Let $K = K_2 - K_1$, then $K \in \mathcal{K}(H)$ is self-adjoint. Hence by \cite[Theorem 5.1, Page 178]{Gohberg}, there exists a system of orthonormal eigenvectors $\phi_1, \phi_2, \phi_3, \dots$ with the corresponding eigenvalues $\lambda_1, \lambda_2, \lambda_3,\dots, $ respectively such that,
	\begin{equation*}
	Kx = \sum_{n=1}^{\infty} \lambda_n \langle x , \phi_n \rangle \phi_n,\ \forall x \in H,
	\end{equation*}
	where $\{\lambda_n\}$ converges to 0, when ${\{\lambda_n}\}$ is an infinite sequence. Also $\overline{span}\{\phi_1, \phi_2,\phi_3, \dots,\} = N(K)^{\perp}$\cite[6.1(a), Page 180]{Gohberg}. Now let $\{\psi_\beta: \beta \in \Lambda\}$ be an orthanormal basis of $N(K)$. Then $\{\phi_n : n \in \mathbb{N}\} \cup \{\psi_\beta : \beta \in \Lambda\}$ forms an orthonormal basis of $H$. Thus we get,
	\begin{equation*}
	Tx = \sum_{n=1}^{\infty} (\alpha + \lambda_n)\langle x , \phi_n \rangle \phi_n + \sum_{\beta \in \Lambda}^{} \alpha \langle x , \psi_\beta \rangle \psi_\beta, \ \forall x \in H.
	\end{equation*}
	Hence $T$ is diagonalizable.
\end{proof}
In the next theorem, we discuss the characterization of positive operators in the $\overline{\mathcal{AN}(H)}$ in terms of the essential spectrum. This result can be viewed as a generalization of \cite[Theorem 2.4]{Rameshpara}.
\begin{theorem}\label{T ess}
	Let $H$ be an infinite dimensional Hilbert space and $T \in \mathcal{B}(H)_{+}$. Then $T \in \overline{\mathcal{AN}(H)}_{+}$  if and only if $\sigma_{ess}(T)$ is a singleton set.
\end{theorem}
\begin{proof}
	Let $T \in \overline{\mathcal{AN}(H)}_+$. By Theorem \ref{T positive}, we have $T = \alpha I - K_1 + K_2$ where, $\alpha \geq 0, \ K_1, K_2 \in \mathcal{K}(H)_+$. Hence by \cite[Proposition 4.2(e), Page 358]{conway} we have, $\sigma_{ess}(T) = \sigma_{ess}(\alpha I) = \{\alpha\}$.

	Conversely, let $\sigma_{ess}(T)$ be a singleton set say $\{\alpha\}$, then $m_e(T) = \alpha$. We have $\sigma_{d}(T) := \sigma(T)\setminus \sigma_{ess}(T)$ contains only isolated eigenvalues of $T$ of finite multiplicity, hence it must be countable. Suppose $\sigma_{d}(T)$ is finite, then  the interval  $[m(T), m_e(T))$ contains only finitely many spectral points. Thus  by \cite[Theorem 2.4]{Rameshpara}, $T \in \mathcal{AN}(H) \subseteq \overline{\mathcal{AN}(H)}$.

If $\sigma_{d}(T)$ is countably infinite, then we have the following cases.
	
	Case $(1)$:	 $[m(T), m_e(T))$ contains only finitely many spectral points;\\
 Then by \cite[Theorem 2.4]{Rameshpara}, $T \in \mathcal{AN}(H)$ which implies $T \in \overline{\mathcal{AN}(H)}$.
	
Case $(2)$: $[m(T), m_e(T))$ contains countably infinite spectral points and $(m_e(T), \|T\|]$ contains only finitely many spectral points;\\
 Let $[m(T), m_e(T)) \cap \sigma_{d}(T) = \{\lambda_1, \lambda_2, \lambda_3,\dots\}$ such that $\lambda_i$'s are increasing to $m_e(T)$. Let $H_1 = \displaystyle{\bigoplus_{i=1}^{\infty}} H_i$ where, $H_i = N(T-\lambda_i I)$.
 Define $T_1:=T|_{H_1} = diag(\lambda_1, \lambda_2, \lambda_3,\dots)$, the diagonal operator. Then $T_1 \in \mathcal{B}(H_1)$ such that $\sigma(T_1) = \{\lambda_1, \lambda_2,\lambda_3,\dots \} \cup \{m_e(T)\} $. Let $(m_e(T), \|T\|] \cap \sigma_{d}(T) = \{\mu_1, \mu_2, \mu_3,\dots, \mu_n\}$ and $H_2 = \displaystyle{\bigoplus_{j=1}^{n}} \tilde{H_j}$ where, $\tilde{H_j} = N(T-\mu_jI)$. Define $T_2 = diag\{\mu_1, \mu_2,\dots,\mu_n\}$. Then $T_2 \in \mathcal{B}(H_2) $ and $\sigma(T_2) = \{\mu_1, \mu_2,\dots,\mu_n\}.$ Since $T$ is diagonalizable, we have $H = H_1 \bigoplus H_2$. Now let $K_1=m_e(T)I_{H_1} - T_1 $ and $F_1=m_e(T)I_{H_2} - T_2$. Then $K_1 \in \mathcal{K}(H_1)$ and $F_1 \in \mathcal{F}(H_2)$. It can be easily verified that $H_1$ and $H_2$ are reducing subspaces for $T$. Hence we get
	\begin{equation*}
	\begin{split}
	T & = \begin{pmatrix}
	T_1 & 0\\ 0 & T_2
	\end{pmatrix}\\
	& = \begin{pmatrix}
	m_e(T)I_{H_1} - K_1 & 0 \\ 0 & m_e(T)I_{H_2} - F_1
	\end{pmatrix}\\
	& = m_e(T) \begin{pmatrix}
	I_{H_1} & 0\\
	0 & I_{H_2}
	\end{pmatrix}
	+ \begin{pmatrix}
	-K_1 & 0\\ 0 & -F_1
	\end{pmatrix}\\
&=m_{e}(T)I+K,
	\end{split}
	\end{equation*}

	where, $K = \begin{pmatrix}
	-K_1 & 0\\ 0 & -F_1
	\end{pmatrix} \in \mathcal{K(H)}$. Hence by Proposition \ref{AN+compact}, $ T \in \overline{\mathcal{AN}(H)}$.

Case $(3)$ $[m(T), m_e(T))$ as well as $(m_e(T), \|T\|]$ contains countably infinite spectral points;\\
 Let $[m(T), m_e(T)) \cap \sigma_{d}(T) = \{\lambda_1, \lambda_2,\lambda_3,\dots\}$ such that $\lambda_i$'s are increasing to $m_e(T)$ and let $H_1 = \displaystyle{\bigoplus_{i=1}^{\infty}} H_i$ where, $H_i = N(T-\lambda_i I)$. Define $T_1$ = diag$(\lambda_1, \lambda_2,\lambda_3,\dots)$. Then $T_1 \in \mathcal{B}(H_1)$ such that $\sigma(T_1) = \{\lambda_1, \lambda_2,\lambda_3\dots\} \cup \{m_e(T)\} $. Now let $(m_e(T), \|T\|] \cap \sigma_{d}(T) = \{\mu_1, \mu_2,\mu_3,\dots\}$ such that $\mu_i$'s are decreasing to $m_e(T)$ and $H_2 = \displaystyle{\bigoplus_{j=1}^{\infty}} \tilde{H_j}$ where $\tilde{H_j} = N(T-\mu_jI)$. Define $T_2 = diag\{\mu_1, \mu_2,\mu_3,\dots\}$. Then $T_2 \in \mathcal{B}(H_2) $ and $\sigma(T_2) = \{\mu_1, \mu_2,\mu_3,\dots\} \cup \{m_e(T)\}.$ Clearly $H = H_1 \bigoplus H_2$. Now, $m_e(T)I - T_1 = K_2$ and $m_e(T)I - T_2 = K_3$ where $K_2 \in \mathcal{K}(H_1)$ and $K_3 \in \mathcal{K}(H_2)$ . Hence we get
	\begin{equation*}
	\begin{split}
	T & = \begin{pmatrix}
	T_1 & 0\\ 0 & T_2
	\end{pmatrix}\\
	& = \begin{pmatrix}
	m_e(T)I_{H_1} - K_2 & 0 \\ 0 & m_e(T)I_{H_2} - K_3
	\end{pmatrix}\\
	& = m_e(T) \begin{pmatrix}
	I_{H_1} & 0\\
	0 & I_{H_2}
	\end{pmatrix}
	+ \begin{pmatrix}
	-K_2& 0\\ 0 & -K_3
	\end{pmatrix}\\
	& = m_e(T) I + K,
	\end{split}
	\end{equation*}
	where, $K = \begin{pmatrix}
	-K_2 & 0\\ 0 & -K_3
	\end{pmatrix} \in \mathcal{K}(H)$. Hence by Proposition \ref{AN+compact}, $ T \in \overline{\mathcal{AN}(H)}$.
\end{proof}

Now, we give the spectral diagram for positive $\mathcal{AN}$-operator and positive operator in $\mathcal{AN}-closure$.  Here for $T \in \mathcal{AN}(H)_+$ (or $\overline{\mathcal{AN}(H)}_{+}$), each stroke represents the spectral points of $\sigma(T)$. These diagrams help us to understand Theorem \ref{T ess} and \cite[Theorem 2.4]{Rameshpara} in a better way and differentiate positive $\mathcal{AN}$-operators from that of positive operators in its closure.
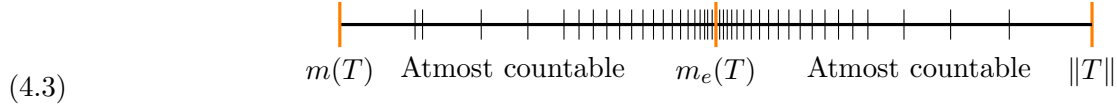
\begin{equation}\label{spectral diagramAN}
\hspace{2cm}
\begin{tikzpicture}
\draw[orange, very thick](5,-.3)--(5,.3);
\filldraw(5,-.3)  node[anchor=north] {$\|T\|$};
\draw[very thick]  (-5,0)--(5,0);
\draw [orange, very thick] (0,-.3)--(0,.3);
\filldraw(0,-.3)  node[anchor=north] {$m_e(T)$};
\draw [orange, very thick](-5,-.3)--(-5,.3);
\filldraw(-5,-.3) node[anchor=north] {$m(T)$};
\draw (-4,-.2)--(-4,.2);
\filldraw(-2.7,-.3)  node[anchor=north] {Finitely many};
\draw (-2.2,-.2)--(-2.2,.2);
\draw (-3,-.2)--(-3,.2);
\draw (-3.5,-.2)--(-3.5,.2);
\draw (-2,-.2)--(-2,.2);
\draw (-1,-.2)--(-1,.2);
\draw (-.3,-.2)--(-.3,.2);
\draw (.1,-.2)--(.1,.2);

\draw (.05,-.2)--(.05,.2);
\draw (.10,-.2)--(.10,.2);
\draw (.15,-.2)--(.15,.2);
\draw (.21,-.2)--(.21,.2);
\draw (.28,-.2)--(.28,.2);
\draw (.37,-.2)--(.37,.2);
\draw (.47,-.2)--(.47,.2);	
\draw (.58,-.2)--(.58,.2);
\draw (.7,-.2)--(.7,.2);
\draw (.83,-.2)--(.83,.2);
\draw (.97,-.2)--(.97,.2);
\draw (1.12,-.2)--(1.12,.2);	
\draw (1.28,-.2)--(1.28,.2);
\draw (1.45,-.2)--(1.45,.2);
\draw (1.63,-.2)--(1.63,.2);
\draw (1.82,-.2)--(1.82,.2);
\draw (2.02,-.2)--(2.02,.2);
\draw (2.5,-.2)--(2.5,.2);
\draw (3.12,-.2)--(3.12,.2);
\draw (3.9,-.2)--(3.9,.2);	
\end{tikzpicture}
\end{equation}
\hspace{5cm}\textbf{Spectral diagram for positive $\mathcal{AN}$-operator}

\vspace{2cm}
{\begin{equation}\label{spectraldiagramANclosure}
\hspace{2cm}
{\begin{tikzpicture}
	\draw[orange, very thick](5,-.3)--(5,.3);
	\filldraw(5,-.3)  node[anchor=north] {$\|T\|$};
	\draw[very thick]  (-5,0)--(5,0);
	\draw [orange, very thick] (0,-.3)--(0,.3);
	\filldraw(0,-.3)  node[anchor=north] {$m_e(T)$};
	\draw [orange, very thick](-5,-.3)--(-5,.3);
	\filldraw(-5,-.3) node[anchor=north] {$m(T)$};
	\filldraw(-2.7,-.3)  node[anchor=north] {Atmost countable};
	\filldraw(2.7,-.3)  node[anchor=north] {Atmost countable};
	\draw (-4,-.2)--(-4,.2);
	\draw (-.05,-.2)--(-.05,.2);
	\draw (-.11,-.2)--(-.11,.2);
	\draw (-.15,-.2)--(-.15,.2);
	\draw (-.21,-.2)--(-.21,.2);
	\draw (-.28,-.2)--(-.28,.2);
	\draw (-.37,-.2)--(-.37,.2);
	\draw (-.47,-.2)--(-.47,.2);	
	\draw (-.58,-.2)--(-.58,.2);
	\draw (-.7,-.2)--(-.7,.2);
	\draw (-.83,-.2)--(-.83,.2);
	\draw (-.97,-.2)--(-.97,.2);
	\draw (-1.12,-.2)--(-1.12,.2);	
	\draw (-1.28,-.2)--(-1.28,.2);
	\draw (-1.45,-.2)--(-1.45,.2);
	\draw (-1.63,-.2)--(-1.63,.2);
	\draw (-1.82,-.2)--(-1.82,.2);
	\draw (-2.02,-.2)--(-2.02,.2);
	\draw (-2.5,-.2)--(-2.5,.2);
	\draw (-3.12,-.2)--(-3.12,.2);
	\draw (-3.9,-.2)--(-3.9,.2);
	
	\draw (.05,-.2)--(.05,.2);
	\draw (.10,-.2)--(.10,.2);
	\draw (.15,-.2)--(.15,.2);
	\draw (.21,-.2)--(.21,.2);
	\draw (.28,-.2)--(.28,.2);
	\draw (.37,-.2)--(.37,.2);
	\draw (.47,-.2)--(.47,.2);	
	\draw (.58,-.2)--(.58,.2);
	\draw (.7,-.2)--(.7,.2);
	\draw (.83,-.2)--(.83,.2);
	\draw (.97,-.2)--(.97,.2);
	\draw (1.12,-.2)--(1.12,.2);	
	\draw (1.28,-.2)--(1.28,.2);
	\draw (1.45,-.2)--(1.45,.2);
	\draw (1.63,-.2)--(1.63,.2);
	\draw (1.82,-.2)--(1.82,.2);
	\draw (2.02,-.2)--(2.02,.2);
	\draw (2.5,-.2)--(2.5,.2);
	\draw (3.12,-.2)--(3.12,.2);
	\draw (3.9,-.2)--(3.9,.2);	
	\end{tikzpicture}}
\end{equation}
\hspace{4cm}\textbf{Spectral diagram for positive operator in $\mathcal{AN}$-closure}}

\begin{remark}\label{property4}
\begin{enumerate}
\item{\label{projinclosure}} If $P \in \mathcal{B}(H)$ is an orthogonal projection with infinite dimensional range and null spaces then, $P \notin \overline{\mathcal{AN}(H)}$ as $\sigma_{ess}(P) =\{0,1\}$.
\item{\label{partialisometry in closure}} if $V$ is a partial isometry such that $N(V)$ and $R(V)$ are infinite dimensional, then $V\notin \overline{\mathcal{AN}(H)}$.
\item if $T\in \overline{\mathcal{AN}(H)}$ and $K\in \mathcal{K}(H)$, then $T+K\in \overline{\mathcal{AN}(H)}$.
\end{enumerate}
\end{remark}

It is clear by the definition that $\mathcal{AN}(H) \subset \mathcal{N}(H)$. But $\overline{\mathcal{AN}(H)} \nsubseteq \mathcal{N}(H)$, if $H$ is an infinite dimensional Hilbert space. This is illustrated in the following example.
\begin{example}
		Let $T :  l^2(\mathbb{N}) \to l^2(\mathbb{N})$ be defined as,
	
	\[ T (e_n)= \left(1-\frac{1}{n}\right)e_n,\  \forall \ n \in \mathbb{N}.\]
	Clearly $T \in \mathcal{B}(l^2(\mathbb{N}))_+$. As 1 is the only limit point of $\sigma(T)$ and there are no eigenvalues with infinite multiplicity, we have $\sigma_{ess}(T) = \{1\}$. Hence by Theorem \ref{T ess}, we get $T \in \overline{\mathcal{AN}(l^2(\mathbb{N}))}$. But,
	\[\|Tx\|^{2} = \sum_{n=1}^{\infty}\left(1-\frac{1}{n}\right)^2 |x_n|^{2} < \sum_{n=1}^{\infty}|x_n|^2 = \|x\|^{2}, \ \forall x = (x_1,x_2,\dots) \in l^2(\mathbb{N}).\]
	Thus $T \notin \mathcal{N}((l^2(\mathbb{N}))$.

Also $\mathcal{N}(H) \nsubseteq \overline{\mathcal{AN}(H)}$. This follows from Remark \ref{property4}(\ref{projinclosure}).
\end{example}

The following example shows that $\overline{\mathcal{AN}(H)}$ is not closed under addition.
\begin{example}
	Let $T :  l^2(\mathbb{N}) \to l^2(\mathbb{N})$ be defined as,
	
	\[ T (e_n)= \begin{cases}
	 e_n & if \ n \ is \ odd,\\
	-e_n & if\ n \ is \ even.
	\end{cases}
	\]
Since $T$ is an isometry, $T \in \overline{\mathcal{AN}(l^2(\mathbb{N}))}$. Clearly the identity operator $I \in \overline{\mathcal{AN}(l^2(\mathbb{N}))}$. Now,
\[ (T + I)(e_n)= \begin{cases}
2 e_n & if \ n \ is \ odd,\\
0 & if\ n \ is \ even.
\end{cases}
\]
Hence $T + I \notin \overline{\mathcal{AN}(l^2(\mathbb{N}))}$ as $\sigma_{ess}(T+I)={\{0,2}\}$.
\end{example}

In general, $T \in \overline{\mathcal{AN}(H)}$ need not imply $T^* \in \overline{\mathcal{AN}(H)}$ which is shown in the following example.
\begin{example}
	Let Let $T :  l^2(\mathbb{N}) \to l^2(\mathbb{N})$ be defined by,
	\begin{equation*}
T((x_1, x_2, x_3,\dots)) = (x_1, 0, x_2, 0, x_3, 0,\dots).
	\end{equation*}
As $T$ is an isometry, $T \in \overline{\mathcal{AN}(l^2(\mathbb{N}))}$.

Now the adjoint of $T$ is given by $T^* :  l^2(\mathbb{N}) \to l^2(\mathbb{N})$ satisfying
\begin{equation*}
T^*((x_1, x_2, x_3,\dots)) = (x_1, x_3, x_5,\dotsb).
\end{equation*}
$T^*$ is a partial isometry with infinite dimensional range and null spaces and  hence by Remark \ref{property4}(\ref{partialisometry in closure}) we get, $T^* \notin \overline{\mathcal{AN}(l^2(\mathbb{N}))}$.
\end{example}
Hence we give a sufficient condition for the adjoint of the operator in $\mathcal{AN}$-closure to be in $\mathcal{AN}$-closure which is similar to the condition for $\mathcal{AN}$-operator given in \cite{Rameshpara}.
\begin{corollary} \label{T*AN-closure}
	Let $H$ be an infinite dimensional Hilbert space and $T \in \mathcal{B}(H)$. If $\sigma_{ess}(T^*T) = \sigma_{ess}(TT^*)$ then, $T \in \overline{\mathcal{AN}(H)}$ if and only if $T^* \in \overline{\mathcal{AN}(H)}$.
\end{corollary}
\begin{proof}
	 If $T \in \overline{\mathcal{AN}(H)}$, then by Theorem \ref{T*T in AN-closure}, $T^*T \in \overline{\mathcal{AN}(H)}$. Now we have, $\sigma_{ess}(T^*T) = \{\alpha\} = \sigma_{ess}(TT^*)$. Hence by using Theorem \ref{T ess},  $TT^* \in \overline{\mathcal{AN}(H)}$. Again by Theorem \ref{T*T in AN-closure}, $T^* \in \overline{\mathcal{AN}(H)}$.
	
	 The converse follows similarly as $T = (T^*)^*$.
\end{proof}

\begin{prop} \label{ess(T*T)=ess(TT^*)}
Let $H$ be an infinite dimensional Hilbert space. If  $T,T^*\in \overline{\mathcal{AN}(H)}$, then $\sigma_{ess}(T^*T)=\sigma_{ess}(TT^*)$.
\end{prop}
\begin{proof}
	By \cite[Theorem 6, Page 173]{Muller}, we have $\sigma_{ess}(T^*T) \setminus \{0\} = \sigma_{ess}(TT^*) \setminus \{0\}$. So it is enough to show, $ 0 \in \sigma_{ess}(T^*T) $ if and only if $0 \in \sigma_{ess}(TT^*)$.
	
	Let $0 \in \sigma_{ess}(T^*T)$. Since $T \in \overline{\mathcal{AN}(H)}$ iff $T^*T \in \overline{\mathcal{AN}(H)}$, we get $\sigma_{ess}(T^*T) = \{0\}$ by using Theorem \ref{T ess}. Hence $T^*T$ is compact which implies $T$ is compact. Therefore, $TT^*$ is compact and $0 \in \sigma_{ess}(TT^*)$. Similarly by using $T^* \in \overline{\mathcal{AN}(H)}$, we can show that if $0 \in \sigma_{ess}(TT^*)$, then $0 \in \sigma_{ess}(T^*T)$.
\end{proof}

Combining Corollary \ref{T*AN-closure} and Proposition \ref{ess(T*T)=ess(TT^*)} we can state the following result. For $\mathcal{AN}$-operators similar result is proved \cite[Proposition 3.3]{NBGR}.

\begin{theorem}\label{2imply3condition}
Let $H$ be an infinite dimensional Hilbert space and $T\in \mathcal{B}(H)$. Then any of the two conditions in the following imply the third one.
\begin{enumerate}
\item $T\in \overline{\mathcal{AN}(H)}$.
\item $T^*\in \overline{\mathcal{AN}(H)}$.
\item $\sigma_{ess}(T^*T)=\sigma_{ess}(TT^*)$.
\end{enumerate}
\end{theorem}

\begin{prop}
	Let $H_1, H_2$ be infinite dimensional Hilbert spaces and $T_i \in \overline{\mathcal{AN}(H_i)}_+$ for $i= 1,2$. Then $T_1 \oplus T_2 \in \overline{\mathcal{AN}(H_1 \oplus H_2)}$ iff $\sigma_{ess}(T_1) = \sigma_{ess}(T_2)$.
\end{prop}
\begin{proof}
	By Theorem \ref{T positive}, we have $T_1 = \alpha I_{H_1} - K_1 + K_2$ and $ T_2 = \beta I_{H_2} - K_3 + K_4$ where, $\alpha, \beta \geq 0,  K_1, K_2\in \mathcal{K}(H_1)_+$ and $K_3, K_4\in \mathcal{K}(H_2)_+$. Now,
	\begin{equation*}
	\begin{split}
	T_1 \oplus T_2 & = \begin{pmatrix}
	\alpha I_{H_1} - K_1 + K_2 & 0 \\ 0 & \beta I_{H_2} - K_3 + K_4
	\end{pmatrix}\\
	& =  \begin{pmatrix}
	\alpha I_{H_1} & 0 \\ 0 & \beta I_{H_2}
	\end{pmatrix}
	-
	\begin{pmatrix}
	K_1 & 0 \\ 0 & K_3
	\end{pmatrix}
	+
	\begin{pmatrix}
	K_2 & 0 \\ 0 & K_4
	\end{pmatrix}.
	\end{split}
	\end{equation*}
	Hence $T_1 \oplus T_2 \in \overline{\mathcal{AN}(H_1 \oplus H_2)}$ iff $\sigma_{ess}(T_1 \oplus T_2)$ is singleton by Theorem \ref{T ess} iff $\alpha = \beta$ iff $\sigma_{ess}(T_1) = \sigma_{ess}(T_2)$.
\end{proof}

\begin{remark}\label{finitedirectsum}
Let $H_j$ be infinite dimensional Hilbert spaces and $T_j\in \mathcal B(H_j)_{+}$ for $j=1,2,\dots,n$ with $\sigma_{ess}(T_j)={\{\alpha}\}$, for some $\alpha\geq 0$. Then $\bigoplus_{j=1}^nT_j\in \overline{\mathcal{AN}(\oplus_{j=1}^nH_j)}$.
\end{remark}
\section{Self-adjoint and normal operators in $\overline{\mathcal{AN}(H)}$}
In this section we give a structure theorem for self-adjoint operators and normal operators in $\overline{\mathcal{AN}(H)}$. All the results in this section generalizes results from \cite{venkuramesh} and \cite{Rameshpara}.

\begin{theorem}
Let $H$ be an infinite dimensional Hilbert space and $T \in \overline{\mathcal{AN}(H)}$ be a self-adjoint operator. Then the essential spectrum of $T$ contains atmost two points. Moreover if $\alpha_1, \alpha_2 \in \sigma_{ess}(T)$, then $\alpha_1 = \pm \alpha_2$.
\end{theorem}
\begin{proof}
	Let $\alpha_j \in \sigma_{ess}(T),$ for $j = 1,2,3$. Then by Theorem \ref{ess spectrum of self-adjoint}, we have either $\alpha_j$ is an eigenvalue of infinite multiplicity or it is the limit point of $\sigma_p(T)$ or $\alpha_j \in \sigma_c(T)$. But by \cite[Proposition 4.4.5, Page 157]{Pedersen}, $\alpha_j \in \sigma_c(T)$ implies $\alpha_j$ is a limit point of $\sigma(T)$. Hence we have the following cases.\\
	\textbf{Case $(1)$} If $\alpha_j \in \sigma_p(T)$ such that $N(T-\alpha_jI)$ is infinite dimensional for each $j=1,2,3$. Since $N(T - \alpha_jI) \subset N(T^2- \alpha^2_jI)$, we have $\alpha^2_j$ is an eigenvalue of $T^2$ for $j= 1,2,3$. Now $T^2 \in \overline{\mathcal{AN}(H)}$ is positive, hence by Theorem \ref{T ess} we get $\alpha^2_1= \alpha^2_2 = \alpha^2_3$. Therefore $\alpha_1 = \pm \alpha_2 = \pm \alpha_3$.\\
	\textbf{Case $(2)$} If any one of $\alpha_j$'s is a limit point of $\sigma(T)$. Without loss of generality let $\alpha_1$ be the limit point of $\sigma(T)$ and $\alpha_2, \alpha_3$ are eigenvalues with infinite multiplicity, then $\alpha^2_1$ is the limit point of $\sigma(T^2)$ and $\alpha^2_2, \alpha^2_3 \in \sigma_p(T^2)$. As $T^2 \in \overline{\mathcal{AN}(H)}$ is positive, by Theorem \ref{T ess} we get $\alpha^2_1= \alpha^2_2 = \alpha^2_3$. Therefore $\alpha_1 = \pm \alpha_2 = \pm \alpha_3$.\\
	\textbf{Case $(3)$} If each $\alpha_j$ is the limit point of $\sigma(T)$, then $\alpha^2_j$ is the limit point of $\sigma(T^2)$ for all $j = 1,2,3$. Since $T^2 \in \overline{\mathcal{AN}(H)}$ is positive, from Theorem \ref{T ess} we get $\alpha^2_1= \alpha^2_2 = \alpha^2_3$. Therefore $\alpha_1 = \pm \alpha_2 = \pm \alpha_3$.
\end{proof}	
%
%

\begin{theorem}\label{T selfadjoint}
	Let $H$ be an infinite dimensional Hilbert space and $T \in \overline{\mathcal{AN}(H)}$ be self-adjoint. Let $T = W|T|$ be the polar decomposition of $T$. Then $T = \alpha W - K_1 + K_2$ where $\alpha \geq 0$, $K_1, K_2$ are self-adjoint compact operators with $K_1K_2 = 0$ and $K^2_1 \leq \alpha^2 I$.
\end{theorem}
\begin{proof}
	Since $T$ is self-adjoint, $W$ is also self-adjoint. Now, as $T \in \overline{\mathcal{AN}(H)}$ we have $|T| \in \overline{\mathcal{AN}(H)}$. By Theorem \ref{T positive}, we have $|T| = \alpha I - \tilde{K_1} + \tilde{K_2}$ where, $\tilde{K_1}, \tilde{K_2} \in \mathcal{K}(H)_+, \alpha \geq 0, \tilde{K_1} \tilde{K_2} = 0$ and $\tilde{K_1} \leq \alpha I$. Then $T = \alpha W -K_1 + K_2$ where, $K_1 = W \tilde{K_1}$ and $K_2 = W \tilde{K_2}$ are compact. We have $W|T| = |T|W$, by \cite[Proposition 6.4, Page 444]{Kubrusly}. Hence we get $W (\tilde{K_2} - \tilde{K_1}) = (\tilde{K_2}- \tilde{K_1}) W$. Since $W$ commutes with $(\tilde{K_2} - \tilde{K_1})$, it also commute with $(\tilde{K_2} - \tilde{K_1})^2$. Therefore we get,
	\begin{align*}
	W (\tilde{K_2} + \tilde{K_1})^2 =W (\tilde{K_2}^2 + \tilde{K_1}^2)= W(\tilde{K_2} - \tilde{K_1})^2
                                    & = (\tilde{K_2} - \tilde{K_1})^2 W \\
                                    &= (\tilde{K_2}^2 + \tilde{K_1}^2) W \\
                                    &= (\tilde{K_2} + \tilde{K_1})^2 W.
	\end{align*}
		As $\tilde{K_2} + \tilde{K_1}$ is positive, we get $W (\tilde{K_1} + \tilde{K_2}) = (\tilde{K_1} + \tilde{K_2}) W$. Thus $W\tilde{K_1} = \tilde{K_1}W$ and $W\tilde{K_2} = \tilde{K_2}W$. Now by using \cite[Problem 5.45(c), Page 428]{Kubrusly}, we get $K^*_1 = K_1$ and $K^*_2 = K_2$. Also, $K_1K_2 = W\tilde{K_1}W\tilde{K_2} = W\tilde{K_1}\tilde{K_2}W = 0$. Finally we have, $K^2_1 \leq \|K_1\|^2 I \leq \alpha^2 I$.
\end{proof}

Next we discuss about a representation of normal operators in $\overline{\mathcal{AN}(H)}$.
\begin{theorem}\label{T normal}
	Let $H$ be an infinite dimensional Hilbert space and $T \in \overline{\mathcal{AN}(H)}$ be normal. Let $T = W|T|$ be the polar decomposition of $T$. Then $T = \alpha W - K_1 + K_2$ where $\alpha \geq 0$, $K_1, K_2$ are normal compact operators, $K_1K_2 = 0$ and $K^*_1 K_1 \leq \alpha^2 I$.
\end{theorem}
\begin{proof}
Normality of $T$ implies the normality of $W$. Since $|T| \in \overline{\mathcal{AN}(H)}$, by Theorem \ref{T positive} we have, $|T| = \alpha I - \tilde{K_1} + \tilde{K_2}$ where, $\tilde{K_1}, \tilde{K_2} \in \mathcal{K}(H)_+, \alpha \geq 0, \tilde{K_1} \tilde{K_2} = 0$ and $\tilde{K_1} \leq \alpha I$. Then $T = \alpha W -K_1 + K_2$ where, $K_1 = W \tilde{K_1}$ and $K_2 = W \tilde{K_2}$ are compact. We have $W|T| = |T|W$, by \cite[Proposition 6.4, Page 444]{Kubrusly}. Hence by following similar steps as in Theorem \ref{T selfadjoint} we get, $W\tilde{K_1} = \tilde{K_1}W$ and $W\tilde{K_2} = \tilde{K_2}W$. Now by taking adjoint on both sides we get, $W^*\tilde{K_1} = \tilde{K_1}W^*$ and $W^*\tilde{K_2} = \tilde{K_2}W^*$. Next we shall prove that $K_1, K_2$ are normal. Consider,
	\begin{equation*}
		K^*_1 K_1 = \tilde{K_1}W^*W\tilde{K_1} = \tilde{K_1}WW^*\tilde{K_1} = W\tilde{K_1}\tilde{K_1}W^* = K_1 K^*_1.
		\end{equation*}
		Hence, $K_1$ is normal. Similarly we can show that $K_2$ is also normal. Also,
\begin{equation*}
K_1 K_2 = W\tilde{K_1}W\tilde{K_2} = W\tilde{K_1}\tilde{K_2}W = 0.
\end{equation*}
		Finally $K^*_1 K_1 = \tilde{K_1}W^*W\tilde{K_1} \leq \|W\|^2 \tilde{K_1}^2 \leq \alpha^2 I$.
\end{proof}
		
\begin{remark}
 From Theorem \ref{T normal}, we have $K_1, K_2$ are normal and
 \begin{equation*}
 K^*_1K_2= \tilde{K_1}W^*W\tilde{K_2} = W^*\tilde{K_1}\tilde{K_2}W = 0
 \end{equation*}
 as $\tilde{K_1}\tilde{K_2} =0$. It is also noted that $K^*_2K_1 = 0$. Therefore we get, 	
\begin{equation*}
	(K_2 -K_1)^* (K_2 - K_1) = K^*_2 K_2 + K^*_1 K_1 = K_2 K^*_2 + K_1 K^*_1 = (K_2 - K_1)(K_2 -K_1)^*.
 \end{equation*}
 Hence $K_2 - K_1$ is normal.
\end{remark}
\begin{theorem}
	Let $T \in \overline{\mathcal{AN}(H)}$ be normal. Then there exists $(H_n, U_n)_{n \in \mathbb{N}}$, where $H_n$ is a reducing  subspace for $T$, $U_n \in \mathcal{B}(H_n)$ is a unitary operator and $\lambda_n \in \sigma(|T|)$ for all $n \in \mathbb{N}$ such that
\begin{enumerate}
\item $H = \displaystyle \bigoplus_{n\in \mathbb{N}}H_n$
\item $T = \displaystyle \bigoplus_{n \in \mathbb{N}}\lambda_n U_n$.
\end{enumerate}

\end{theorem}
\begin{proof}
	Let $T = U|T|$ be the polar decomposition of $T$.
Since $T \in \overline{\mathcal{AN}(H)}$, we have $|T| \in \overline{\mathcal{AN}(H)}$. By Theorem \ref{T ess}, we have $\sigma_{ess}(|T|)$ is singleton set, say $\{\lambda\}$. Then $\lambda = m_e(T)$. We shall consider the following cases.

\textbf{Case $(1)$} $m_e(T) = \|T\|$:\\
Since $\sigma(|T|)$ is countable, let $\sigma_{d}(|T|) = \{\lambda_1, \lambda_2,\lambda_3,\dots\}$ be such that $\{\lambda_i\}$ is increasing to $\|T\|$. Consider $H_i = N(|T|- \lambda_iI)$ for all $i \in \mathbb{N}$. As $|T|$ is diagonalizable, we have $H = \displaystyle{\bigoplus_{n \in \mathbb{N}}}
H_n$ and $|T| = \displaystyle{\bigoplus_{n \in \mathbb{N}}} \lambda_n I_n$ where, $I_n : H_n \to H_n$ is the identity operator with $\sigma(|T|) = \{\lambda_n\}_{n \in \mathbb{N}} \cup \{\lambda\}$.

We first show that each $H_i$ reduces $U$. If $x \in H_i$, then $|T|x-\lambda_ix =0$.  As $T$ is normal, we have $U|T| = |T|U$ by \cite[Proposition 6.4]{Kubrusly}. Now $U(|T|x-\lambda_ix) =0$ implies that $|T|Ux - \lambda_i Ux = 0$. Hence $Ux \in H_i$ for all $i \in \mathbb{N}$. Thus $H_i$ is invariant under $U$. Similarly, $U^*|T| = |T|U^*$ implies that $H_i$ is invariant under $U^*$. Thus $H_i$ reduces $U$ for all $i \in \mathbb{N}$. Since the restriction of partial isometry to the reducing subspace is a partial isometry \cite[Lemma 1]{HalmosWallen}, we get $T = \displaystyle{\bigoplus_{n \in \mathbb{N}}}\lambda_n U_n$ where $U_n : H_n \to H_n$ is a partial isometry.
Let $x \in N(|T|-\lambda_i I)$ with $\lambda_i \neq 0$ then, $|T|x = \lambda_i x$. For any $y \in N(U)$,
\begin{equation*}
\langle y,x \rangle = \langle y, \frac{|T|x}{\lambda_i} \rangle = \frac{1}{\lambda_i}(\langle y,|T|x \rangle) = 0,
\end{equation*}
as $|T|x\in R(|T|)\subseteq N(|T|)^{\bot}=N(U)^{\bot}$.
 Hence for $\lambda_i \neq 0$, we have $N(|T|- \lambda_i I) \subset N(U)^\perp$. Hence $U_n : H_n \to H_n$ is an isometry. Since $\lambda_n \in \sigma_{d}(T)$, each $H_n$ is finite dimensional. Thus $U_n$ is a unitary operator.
 If $\lambda_{n_0}=0$ for some $n_0\in \mathbb {N}$, then clearly $H_{n_0}=N(T)$ and $U_{n_0}=0$. In this case also, normality of $T$ imply that $H_{n_0}$ is reducing subspace for $T$.

If $m_e(T) = \lambda$ is an eigenvalue with infinite multiplicity, then $U|_{N(|T|- \lambda I)} = U_{\lambda}$ is an isometry on $N(|T|- \lambda I)$. Since normality of $T$ implies $U$ is normal, hence $U_\lambda$ is unitary. In this case, we get $H = \displaystyle{\bigoplus_{n \in \mathbb{N}}}H_n \bigoplus H_\lambda$ and $T =  \displaystyle{\bigoplus_{n \in \mathbb{N}}}\lambda_n U_n \bigoplus \lambda U_\lambda$.

 \textbf{Case $(2)$} $m_e(T) = m(T)$:\\
   In this case, $[m(T),m_e(T))= \{m_e(T)\}$. By \cite[Theorem 2.4]{Rameshpara} $|T| \in \mathcal{AN}(H)$, hence $T \in \mathcal{AN}(H)$. So by using  \cite[Theorem 3.9]{Rameshpara} we get $H = \displaystyle{\bigoplus_{n \in \mathbb{N}}} H_n$ and $ T = \displaystyle{\bigoplus_{n \in \mathbb{N}}}\lambda_n U_n$ where, $H_n$ reduces $T$, $U_n \in \mathcal{B}(H_n)$ is a unitary operator and $\lambda_n \in \sigma(|T|)$.\\

 \textbf{Case $(3)$} $m(T) < m_e (T) < \|T\|$:\\
If $[m(T), m_e(T))$ contains only finitely many spectral points then, by \cite[Theorem 2.4]{Rameshpara} $T \in \mathcal{AN}(H)$. Hence the proof follows from Case $(2)$.

 If $(m_e(T), \|T\|]$ contains only finitely many spectral points then, by \cite[Theorem 3.10]{NeeruRamesh} $T \in \mathcal{AM}(H)$. Hence by \cite[Theorem 4.4]{NeeruRamesh} we get the required result.

 Next we assume that  both $[m(T), m_e(T))$ and $(m_e(T), \|T\|]$ contains infinitely many eigenvalues of $|T|$. Let $[m(T), m_e(T)) \cap \sigma(|T|) = \{\lambda_1, \lambda_2,\lambda_3,\dots\}$ and $(m_e(T), \|T\|] \cap \sigma(|T|) = \{\mu_1, \mu_2, \mu_3,\dots\}$, where $\{\lambda_i\}$ increases to $m_e(T)$ and $\{\mu_i\}$ decreases to $m_e(T)$.

 Let $H_i = N(|T|- \lambda_i I)$ and $\tilde{H_j} = N(|T|- \mu_jI)$ for all $i,j \in \mathbb{N}$. Then clearly $H = \displaystyle{\bigoplus_{i \in \mathbb{N}}}N(|T|- \lambda_i I) \displaystyle{\bigoplus_{j \in \mathbb{N}}} N(|T|- \mu_j \tilde{I})$ and $|T| = \displaystyle{\bigoplus_{i \in \mathbb{N}}}\lambda_i I_i \displaystyle{\bigoplus_{j \in \mathbb{N}}} \mu_j \tilde{I_j}$ where, $I_i : H_i \to H_i$ and $\tilde{I_j} : \tilde{H_j} \to \tilde{H_j}$ are identity operators with $\sigma(|T|) = \{\lambda_i\}_{i \in \mathbb{N}} \cup \{\mu_j\}_{j \in \mathbb{N}} \cup \{m_e(T)\}$. Since $N(|T|- \lambda_iI)$ and $N(|T|- \mu_j I)$ reduces $U$ for all $i,j \in \mathbb{N}$ we get,
 \begin{equation*}
 T = \bigoplus_{i \in \mathbb{N}}\lambda_iU_i \bigoplus_{j \in \mathbb{N}}\mu_j \tilde{U_j}.
 \end{equation*}
 As $N(|T|-\lambda_i I) \subset N(U)^\perp$ and $ N(|T|- \mu_j I) \subset N(U)^\perp$, whenever $\lambda_i\neq 0$ and $\mu_j\neq 0$, we get $U_i : H_i \to H_i$ and $\tilde{U_j} : \tilde{H_j} \to \tilde{H_j}$ are unitary operators. If one of $\lambda_i$ or $\mu_j$ is zero, then the corresponding eigenspace is nothing but $N(T)$, which reduces $T$ as $T$ is normal.

 If $m_e(T) = \lambda$ is an eigenvalue with infinite multiplicity, then $U|_{N(|T|- \lambda I)} = U_{\lambda}$ is unitary on $N(|T|- \lambda I)$, as $T$ is normal. In this case, we get
   $H = \displaystyle{\bigoplus_{i \in \mathbb{N}}}N(|T|- \lambda_i I) \displaystyle{\bigoplus_{j \in \mathbb{N}}} N(|T|- \mu_j \tilde{I}) \bigoplus N(|T| - \lambda I)$ and
  $ T = \bigoplus_{i \in \mathbb{N}}\lambda_iU_i \bigoplus_{j \in \mathbb{N}}\mu_j \tilde{U_j} \bigoplus \lambda U_\lambda$.

\end{proof}

\section{Closure of $\mathcal{AM}$-operators}
It is well known that the set of all $\mathcal{AM}$-operators is distinct from that of $\mathcal{AN}$-operators, as the set of all $\mathcal{AN}$-operators contains the algebra of compact operators but the only compact operators  that are minimum attaining are the non-injective ones \cite[Proposition 1.3]{CarvajalNevesAM}. In this section, we shall show that the operator norm closure of these two sets of operators are the same.

We first recall the structure theorem for positive $\mathcal{AM}$-operators.
\begin{theorem}\cite[Theorem 5.8]{GRS} \label{positive AM}
	Let $H$ be an infinite dimensional Hilbert space and $T \in \mathcal{B}(H)$. Then the following are equivalent:
	\begin{enumerate}
	\item  $T \in \mathcal{AM}(H)_+$.
	\item  There exists a decomposition for $T$ of the form $T = \alpha I - K + F$ where $K\in \mathcal{K}(H)_{+}$  with $\|K\| \leq \alpha$ and $F\in \mathcal{F}(H)_{+}$  satisfying $KF = F K = 0$. Moreover, this decomposition is	unique.
	\end{enumerate}
\end{theorem}

The following result can be proved easily.
\begin{theorem}\label{newcharAMoperators}
Let $H$ be an infinite dimensional Hilbert space and $T\in \mathcal{B}(H)_{+}$. Then
$T\in \mathcal{AM}(H)_{+}$  if and only if there exists $\beta \geq 0$ such that $(T-\beta I)^{+}\in \mathcal{F}(H)_{+}$, $(T-\beta I)^{-}\in \mathcal{K}(H)_{+}$ and $\|(T-\beta I)^{-}\|\leq \beta$.
\end{theorem}

\begin{prop} \label{AM+finiterank}
	Let $T \in \mathcal{AM}(H_1, H_2)$ and $F \in \mathcal{F}(H_1, H_2)$. Then $T + F \in \mathcal{AM}(H_1, H_2)$.
\end{prop}
\begin{proof}
	Let $A = T + F$. Consider $A^*A = (T^*+ F^*)(T +F) = T^*T + \tilde{F}$ where, $\tilde{F} = T^*F + F^*T + F^*F \in \mathcal{F}(H_1)$ and ${\tilde{F}}^* = \tilde{F}$. By using \cite[Theorem 5.14]{GRS} and Theorem \ref{positive AM}, we have $T^*T = \alpha I - K_1 +F_1$ where, $\alpha \geq 0, K_1 \in \mathcal{K}(H_1)_+, F_1 \in \mathcal{F}(H_1)_+, \|K_1\| \leq \alpha$ and $K_1F_1 = 0$. Now $A^*A = \alpha I - K_1 + F_2 $ where, $F_2 = F_1 + \tilde{F} \in \mathcal{F}(H_1)$ and $F^*_2 = F_2$. Thus $A^*A \in \mathcal{AM}(H_1)$ \cite[Theorem 5.6]{GRS}. Hence by \cite[Theorem 5.14]{GRS}, we get $A \in \mathcal{AM}(H_1, H_2)$.
	\end{proof}
\begin{corollary}\label{AM+compact}
	Let $T \in \mathcal{AM}(H_1, H_2)$ and $K \in \mathcal{K}(H_1, H_2)$. Then $T + K \in \overline{\mathcal{AM}(H_1, H_2)}$.
\end{corollary}
\begin{proof}
	Since $K \in \mathcal{K}(H_1, H_2)$, there exists a sequence of finite rank operators say $\{F_n\}$ such that $F_n \to K$ in the operator norm. So $T + F_n \to T+K$ in the operator norm. As $T + F_n \in \mathcal{AM}(H_1, H_2)$ [Proposition \ref{AM+finiterank}] for all $n \in \mathbb{N}$, we get $T+ K \in \overline{\mathcal{AM}(H_1, H_2)}$.
\end{proof}

\begin{corollary}\label{parisoAM+com}
If $W\in \mathcal{AM}(H_1,H_2)$ is a partial isometry, $K\in \mathcal{K}(H_1,H_2)$ and $\beta\geq 0$, then $\beta W+K\in \overline{\mathcal{AM}(H_1,H_2)}$.
\end{corollary}

\begin{theorem}\label{AM closure}
	Let $T \in \overline{\mathcal{AM}(H_1, H_2)}$ and $T = W|T|$ be the polar decomposition of $T$. Then $T = \beta W + K$, for some $K \in \mathcal{K}(H_1, H_2)$ and $\beta \geq 0$.
\end{theorem}
\begin{proof}
	As $T \in \overline{\mathcal{AM}(H_1, H_2)}$, there exists a sequence of $\mathcal{AM}$-operators, $\{T_n\}$ such that $T_n \to T$ in the operator norm. Then $|T_n| \to |T|$ as $n \to \infty$ by \cite[Problem 15, Page 217]{ReedSimon1}. As $T_n \in \mathcal{AM}(H_1, H_2)$ we have $|T_n| \in \mathcal{AM}(H_1)$, from Theorem \ref{positive AM}, we get $|T_n|= \beta_n I - K_n + F_n$, where $\beta_n \geq 0, \ K_n \in \mathcal{K}(H_1)_+, F_n \in \mathcal{F}(H_1)_+,$ with $K_nF_n = 0, \|K_n\| \leq \beta_n $ for all $n \in \mathbb{N}$. As $\beta_n\leq \||T_n|\|=\|T_n\|<\sup_{n}\|T_n\|<\infty$,  $\{\beta_n\}$ is bounded. Hence there exists a subsequence say, $\{\beta_{n_k}\}$ such that $\beta_{n_k} \to \beta$ as $k \to \infty$ for some $\beta \geq 0.$ Then $K_{n_k} + F_{n_k} \to K_1$, for some $K_1 \in \mathcal{K}(H_1)$. Hence $|T_{n_k}| \to \beta I + K_1$ as $k \to \infty$. This implies $|T| = \beta I + K_1$. Thus $T = \beta W + K$, where $K = W K_1 \in \mathcal{K}(H_1, H_2)$.
\end{proof}

\begin{corollary}\label{AMprop}
	Let $T\in \overline{\mathcal{AM}(H_1,H_2)}$ and $T\notin \mathcal{K}(H_1,H_2)$ . Then
	\begin{enumerate}
		\item  \label{AMfdmlkernel}$N(T)$ is finite dimensional
		\item  $R(T)$ is closed
		\item  $T$  is left semi-Fredholm operator.
	\end{enumerate}
\end{corollary}

\begin{proof}
	From Theorem \ref{AM closure}, we have
	\begin{equation}\label{AMpropeq1}
	|T| = \beta I + K_1,\, \text{ for some}\; \beta \geq 0\; \text{and}\; K_1 \in \mathcal{K}(H).
	\end{equation}
	Hence following the similar steps as in Corollary \ref{properties2}, we get the required results.
\end{proof}
\begin{remark} \label{CompactinAMclos}
	If $T \in \mathcal{K}(H)$, Then $T \in \overline{\mathcal{AM}(H)}$. But in general, a compact operator need not be in $\mathcal{AM}(H)$. In fact, a compact operator is minimum attaining if and only if it is not one-to-one.
\end{remark}
 
 \begin{prop}
 If $W\in \mathcal B(H_1,H_2)$ is a partial isometry, then  $W \in \mathcal{AM}(H_1,H_2)$ if and only if $W\in \overline{\mathcal{AM}(H_1,H_2)}$.
 \end{prop}
	\begin{proof}
		Let $W \in \mathcal{AM}(H_1, H_2)$. Then clearly $W \in \overline{\mathcal{AM}(H_1, H_2)}$.
		
	On the other hand let  $W \in \overline{\mathcal{AM}(H_1,H_2)}$. If $W \in \mathcal{K}(H_1, H_2)$, then $W \in \mathcal{F}(H_1,H_2)$ which further implies $W \in \mathcal{AM}(H_1,H_2)$. If $W \notin \mathcal{K}(H_1, H_2)$, then by Corollary \ref{AMprop}(\ref{AMfdmlkernel}) we have, $N(W)$ is finite dimensional. Hence by \cite[Theorem 3.10]{CarvajalNevesAM}, $W \in \mathcal{AM}(H_1,H_2)$.
	\end{proof}

\begin{theorem}\label{equalityofclosures}
$\overline{\mathcal{AN}(H_1,H_2)}=\overline{\mathcal{AM}(H_1.H_2)}$.
  \end{theorem}
  \begin{proof}
   Let $T \in \overline{\mathcal{AN}(H_1,H_2)}$  with the polar decomposition  $T = W|T|$. Then by Theorem \ref{general AN-closure structure}, we have $T = \alpha W +K$ for some $\alpha \geq 0$ and $K \in \mathcal{K}(H_1, H_2)$. If $T \in \mathcal{K}(H_1, H_2)$, then $\alpha =0$ and hence $T \in \overline{\mathcal{AM}(H_1, H_2)}$ [Remark \ref{CompactinAMclos}]. If $T \notin \mathcal{K}(H_1, H_2)$ then by Corollary \ref{properties2}(\ref{fdmlkernel}), we have $N(T)$ is finite dimensional so $N(W)$ is finite dimensional and hence by \cite[Theorem 3.10]{CarvajalNevesAM}, we get $W \in \mathcal{AM}(H_1,H_2)$. Therefore $T = \alpha W +K \in \overline{\mathcal{AM}(H_1, H_2)}$ by Corollary \ref{parisoAM+com}.
   
  Conversely  let $T \in \overline{\mathcal{AM}(H_1,H_2)}$ and $T = W|T|$ be the polar decomposition of $T$. Then by Theorem \ref{AM closure}, we have $T = \beta W +K$ for some $\beta \geq 0$ and $K \in \mathcal{K}(H_1, H_2)$. If $T \in \mathcal{K}(H_1, H_2)$, then $T \in \overline{\mathcal{AN}(H_1, H_2)}$. If $T \notin \mathcal{K}(H_1, H_2)$ then by Corollary \ref{AMprop}(\ref{AMfdmlkernel}), we have $N(T)$ is finite dimensional so $N(W)$ is finite dimensional and hence by \cite[Proposition 3.14]{CarvjalNeves}, we get $W \in \mathcal{AN}(H_1,H_2)$. Therefore by Corollary \ref{partialisometryAN+compact}, $T = \beta W +K \in \overline{\mathcal{AN}(H_1, H_2)}$.
   \end{proof}
	 Hence we have the following results for the operators in $\overline{\mathcal{AM}(H)}$ which are similar to the operators in $\overline{\mathcal{AN}(H)}$.
	\begin{theorem}
			Let $T \in \mathcal{B}(H_1,H_2)$. The following results hold true.
				\begin{enumerate}
					\item $T \in \overline{\mathcal{AM}(H_1,H_2)}$
					\item $|T| \in \overline{\mathcal{AM}(H_1)}$
				    \item 	$T^*T \in \overline{\mathcal{AM}(H_1)}$.					
				\end{enumerate}
\end{theorem}

\begin{theorem}
Let $H$ be an infinite dimensional Hilbert space and $T\in \mathcal B(H)_{+}$. Then the following results hold true.
\begin{enumerate}
			\item If $T \in \overline{\mathcal{AM}(H)}_+$, then $T$ is diagonalizable.	
			\item $T \in \overline{\mathcal{AM}(H)}_+$ iff there exists $\beta \geq 0, K_1,K_2 \in \mathcal{K}(H)_+$ such that $ T = \beta I - K_1 + K_2$ with $ \|K_1\| \leq \beta$ and $K_1 K_2 = 0$.
			\item Let $T \in \mathcal{B}(H)$ be positive. Then $T \in \overline{\mathcal{AM}(H)}$ if and only if $\sigma_{ess}(T)$ is a singleton set.
\end{enumerate}
\end{theorem}

\begin{theorem}
		Let $H$ be an infinite dimensional Hilbert space and $T \in \mathcal{B}(H)$. Then any of the following two conditions imply the third one.
	\begin{enumerate}
		\item $T \in \overline{\mathcal{AM}(H)}$.
		\item $T^* \in \overline{\mathcal{AM}(H)}$.
		\item $\sigma_{ess}(T^*T)=\sigma_{ess}(TT^*)$.
	\end{enumerate}
\end{theorem}

\begin{theorem}
 Let $T \in \overline{\mathcal{AM}(H)}$ be normal. Then there exists $(H_n, U_n)_{n \in \mathbb{N}}$, where $H_n$ is reducing a subspace for $T$, $U_n \in \mathcal{B}(H_n)$ is a unitary operator and $\lambda_n \in \sigma(|T|)$ for all $n \in \mathbb{N}$ such that
			\begin{itemize}
				\item[(a)] $H = \displaystyle \bigoplus_{n \in \mathbb{N}}H_n$
				\item[(b)] $T = \displaystyle \bigoplus_{n \in \mathbb{N}}\lambda_n U_n$.
			\end{itemize}
	\end{theorem}

We have already studied the structure of normal operators in $\overline{\mathcal{AN}(H)}$. A lot of studies are done about the properties of non-normal operators and one of the important classes of non-normal operators are hyponormal operators i.e. $T^*T - TT^* \geq 0$. So we end this section with the following question.

\begin{question}
What is the representation of $T \in \overline{\mathcal{AN}(H)}$, when $T$ is hyponormal?
\end{question}

\section*{Acknowledgement}
The first author's work is supported by SERB Grant No. MTR/2019/001307,
Govt. Of India. The second author is thankful for the financial support from the  Department of Science and Technology - INSPIRE Fellowship (Grant No. DST/INSPIRE FELLOWSHIP/2018/IF180107). The authors thank Prof. Hiroyuki Osaka for his suggestions on the manuscript.

\bibliographystyle{amsplain}

\end{document}